\documentclass[12pt,a4paper]{amsart}
\usepackage{amsmath}
\usepackage{amsthm}
\usepackage{amsfonts}
\usepackage{amssymb}
\numberwithin{equation}{section}
\newtheorem{theorem}{Theorem}[section]

\newtheorem{lemma}[theorem]{Lemma}
\newtheorem{prop}[theorem]{Proposition}
\theoremstyle{definition}
\newtheorem{defi}{Definition}[section]
\newtheorem{rem}[defi]{Remark}

\newcommand\dd{\partial}

\newcommand\half{\frac{1}{2}}

\newcommand\ov{\overline}

\newcommand\rhat{\widehat\rho}

\newcommand\w{\wedge}
\newcommand\g{\mathfrak g}
\newcommand\m{\mathfrak m}

\newcommand\h{\mathfrak h}
\newcommand\ha{\widehat{\mathfrak h}}
\newcommand\pp{\mathfrak p}

\newcommand\bb{\mathfrak b}\newcommand\D{\Delta}
\renewcommand\l{\lambda}\newcommand\Dp{\Delta^+}

\newcommand\Da{\widehat\Delta}
\newcommand\Pia{\widehat\Pi}
\newcommand\Dap{\widehat\Delta^+}
\newcommand\Wa{\widehat{W}}
\renewcommand\d{\delta}

\renewcommand\t{\otimes}
\renewcommand\a{\alpha}
\renewcommand\aa{\mathfrak a}
\renewcommand\b{\bb}

\renewcommand\k{\mathfrak k}
\newcommand\q{\mathfrak q}

\renewcommand\i{{\mathfrak i}}

\newcommand\ganz{\mathbb Z}

\newcommand\s{\sigma}
\renewcommand\L{\Lambda}

\renewcommand\aa{\mathfrak a}

\renewcommand\u{\mathfrak u}

\newcommand\C{\mathbb C}
\newcommand\R{\mathbb R}
\newcommand\si{\sigma}

\renewcommand\ha{\widehat{\mathfrak h}}

\newcommand{\germ}[1]{\mathfrak{#1}}
\newcommand\p{\Phi}

\newcommand{\wL}{{\widehat L}}

\begin{document}

\title{Casimir operators, abelian subspaces and $\u$-cohomology} 
\author{ Pierluigi M\"oseneder Frajria\\ Paolo Papi}
\keywords{Lie algebra, Casimir operator, abelian subspace, $\u$-cohomology}
\subjclass{17B20, 17B56}
\maketitle
\section{Introduction} 
This note is an exposition of old and recent results of B. Kostant 
regarding the relationships between
  the exterior algebra of a simple Lie algebra $\g$ and the
action of the Casimir operator on it (see \cite{K1}, \cite{K2}, \cite{Kimrn}, \cite{Kinv}). A key role in this connection 
is played by the abelian subalgebras
  of $\g$ and in particular by the abelian ideals of a Borel
subalgebra $\mathfrak b$ of $\g$. These objects have been intensively and 
thoroughly investigated after
nice results of D. Peterson and subsequent work of several authors 
which link these ideals  to
discrete series, the theory of affine Weyl groups,  
combinatorics and number theory. \par
The previous setting can be extended to a $\ganz_2$-graded Lie algebra
$\g=\g^{\ov 0}\oplus \g^{\ov{1/2}}$, where the role of abelian subalgebras is played by 
the abelian subspaces of $\g^{\ov{1/2}}$ (here $\g^{\ov 0},\,\g^{\ov{1/2}}$ denote the sets of fixed and
antifixed points of the involution $\s$ on $\g$ inducing the $\ganz_2$-grading). 
In the following we will refer to this more general setting as the {\it graded case}. The framework we have described at the beginning  will be called the {\it adjoint case} (it is indeed a particular instance of the graded case: see Section 4). \par
The generalization of the results
of Kostant to the
graded case can be found in
\cite{P},
\cite{cmp},
\cite{Han}.\par Recently (cf. \cite{Kinv}) Kostant pointed out a connection between his old results on
abelian subalgebras and the generalization to the affine case due to Garland-Lepowsky \cite{Ga} of his
classical  results on $\u$-cohomology \cite{Kveryold}. In this paper, we exploit this connection to obtain a
unified approach to Kostant's results and their graded generalizations. 
 One of the advantages of our approach consists in 
avoiding  any reference to the theory
of Clifford algebras. Another useful device that we introduce in our 
treatment is  the natural  isomorphism
$$\w^p\g^{\ov{1/2}} {\buildrel \cong \over \longrightarrow}\ \w^{(p,p/2)}\u^-$$
where $\u^-=t^{-1}\g^{\ov 0}[t^{-1}]\oplus t^{-\half}\g^{\ov{1/2}}[t^{-	\half}]$  (see \eqref{grado} for undefined notation). This  explains the role played  by affine Lie algebras and their cohomology. 
\par
Although the main results are individually known (we have tried to make precise attributions
in Section 4), the new feature of our approach
consists in exploting formula  \eqref{lap}, which relates 
the action on $\w\u^-$ of the Casimir operator of $\k$, the Laplacian 
and the scaling element  of the affine Lie algebra $\widehat L(\g,\s)$.
This formula is new in the graded case and it is known as Garland's formula in the adjoint case. The connection between 
Garland's formula and abelian ideals theory has been noticed 
by Kostant in \cite{Kinv}. \par
 Formula \eqref{lap} is the cornerstone of the present work, for it allows us to give a clean, compact  and unified treatment of the various contributions to the subject.
The exposition is basically self-contained, with two exceptions: a ``Laplacian calculation" which can be found in 
\cite{Kumar} and a technical lemma which is taken from \cite{kac-2007}.
The main results are Theorems \ref{eigen}, \ref{w}, \ref{finito}. 
\section{Setup}
Let $\g$ be a complex semisimple Lie algebra and let $(\cdot,\cdot)$  be its Killing form. Let $\s$ be an involutive
automorphism of $\g$. If $j\in\R$ set $\bar j=j+\ganz\in\R/\ganz$. Let $\g^{\bar j}$  be the $e^{2\pi i j}$-eigenspace of
$\s$, so that we can write $\g=\k\oplus\pp$ where $\k=\g^{\bar0}$ and $\pp=\g^{\ov{1/2}}$.  Let $n$ be the rank of $\k$ and
$N$ its dimension. Fix a Borel subalgebra $\bb_0$ of $\k$, with Cartan component $\h_0$, and let $\Dp_0$ be the set of
positive roots of the root system $\D_0$ of $\k$ corresponding to the previous choice.

Let $\wL(\g,\s)$ be the twisted affine Kac-Moody Lie algebra corresponding to $\g$ and $\s$
(cf. \cite{Kac}). More precisely introduce a 
Cartan subalgebra $\ha$ by setting
$$\h'=\h_0\oplus\C c,\qquad
\ha=\h_0\oplus\C c\oplus\C d$$
and define
\begin{align*}
&L(\g,\s)=\sum_{j\in\ganz}(t^j\otimes\k)\oplus\sum_{j\in\half +\ganz}(t^j\otimes\pp),\\
&L'(\g,\s)=L(\g,\s)+\h',\\
&\wL(\g,\s)=L(\g,\s)+\ha.\end{align*}

If $x\in\g^{\bar r}$, we set $x_r=t^r\otimes x$ for any $r\in\bar r$. With this notation 
the bracket of $\wL(\g,\s)$ is defined by 
$$
[x_r+ac+bd,x'_s+a'c+b'd]=[x,x']_{s+r}+sb x'_s+rb'x_r+\d_{r,-s}r(x,x')c
$$
for $a,a',b,b'\in\C$. Let $\Da$ denote
the set of roots of $\wL(\g,\s)$ and $$\Dap=\Dp_0\cup\{\a\in\Da\mid\a(d)>0\}.$$ 
Then $\Dap$ is a set of positive roots for
$\Da$. We let
$
\Pia
$
be the corresponding set of simple roots.   It follows from \cite[Exercise 8.3]{Kac} 
that $\Pia$ is a finite linearly independent subset of $\ha^*$ with exactly $n+1$ elements.  
We set $\Pia=\{\alpha_0,\a_1,\dots,\a_n\}$.
\vskip5pt
If $\l\in \ha^*$, we denote by  $\bar \l$  the restriction of $\l$ to $\h_0$. Define $\delta\in\ha^*$
 by setting $\delta(\h_0)=\d(c)=0$ and $\d(d)=1$. 
 It is easy to check that
 $(\cdot,\cdot)$ is nondegenerate when restricted to $\h_0$. Thus for $\mu\in\h_0^*$ we can define
$h_{\mu}$ to be the unique element of $\h_0$ such that
$\mu(h)=(h_{\mu},h)$ for all $h\in\h_0$. Then one can define a bilinear form on $\h_0^*$ -- still denoted by $(\cdot,\cdot)$
--  by setting $(\mu,\eta)=(h_\mu,h_\eta)$. 
\vskip5pt
 Write $\a_i=s_i\d+\overline\a_i$. By  \cite[Exercise 8.3]{Kac} we have that $\bar\a_i\ne0$. Set
$h_i=\frac{2}{(\overline\a_i,\overline\a_i)}h_{\overline\a_i}$ and fix
$e_i=t^{s_i}\t X_i$, $f_i=t^{-s_i}\t Y_i$ in the root spaces 
$\widehat L(\g,\si)_{\a_i},\widehat L(\g,\si)_{-\a_i}$ respectively, in such a way that
$(X_i,Y_j)=\d_{i,j}\frac{2}{(\overline\a_i,\overline\a_i)}$. Then
$[X_i,Y_j]=\d_{i,j}h_i$. Set
$\a_i^\vee=\frac{2s_i}{(\overline\a_i,\overline\a_i)}c+h_i$ and
$\Pia^\vee=\{\a_0^\vee,\ldots,\a_n^\vee\}$. 
 It follows that $ [e_i,f_j]=\d_{i,j}\a^\vee_i. $
\vskip10pt
Denote by  $\h_\R$   the real span of $\a^\vee_0,\dots,\a_n^\vee$ and let
 $\wL(\g,\s)_\R$  be the real algebra
generated by $\h_\R\oplus\R d$ together with  the Chevalley generators $e_i,f_i,\,1\leq i\leq n$.
 Let $conj$  be the conjugation of $\wL(\g,\s)$ corresponding
to the real form $\wL(\g,\s)_\R$ and define the conjugate linear antiautomorphism
$\sigma_o$  of $\wL(\g,\s)$
 by setting $\sigma_o(h)=conj(h)$, $\sigma_o(e_i)=f_i$, and
$\sigma_o(f_i)=e_i$. We extend the form $(\cdot,\cdot)$ to $\widehat L(\g,\s)$ by setting 
$$
(x_r,y_s)=\d_{r,-s}(x,y),\,(L(\g,\s),d)=(L'(\g,\s),c)=(d,d)=0,\,(c,d)=1.
$$
It is easy to check that $(\cdot,\cdot)$ is a nondegenerate invariant form on $\wL(\g,\s)$. In particular, it is nondegenerate when restricted to $\ha$.
We let $\nu:\ha\to\ha^*$ be the isomorphism induced by $(\cdot,\cdot)$, i.e.\ $\nu(h)(k)=(h,k)$.
\par

Since $(\cdot,\cdot)$ is real on $\h_\R$, we have that $(\g_\R,\g_\R)\subset\R$.
Following \cite[Definition~2.3.9]{Kumar}, we can therefore define the  Hermitian form $\{\cdot ,\cdot \}$ on
$\wL(\g,\s)$ by setting
\begin{equation}\label{sigma}
\{x,y\}=(x,\sigma_o(y)).
\end{equation}
This form is contravariant, i.e. $\{[a,x],y\}=-\{x,[\s_o(a),y]\}$.
\vskip5pt
We set
\begin{align*}
&\m= \k+\ha,\\
&\u= \sum_{\a(d)>0}\wL(\g,\s)_\a,\\
&\q=\mathfrak m \oplus \u.\\
\end{align*} 
 We also set $ \u^-=\sum\limits_{\a(d)<0}\wL(\g,\s)_\a,\,
\q^-=\m\oplus \u^-$; note that 
 $\s_o(\u)=\u^-$. Since $(\u,\q)=0$ and the form
$(\cdot,\cdot)$ is  nondegenerate on $\wL(\g,\s)$, it follows that $\{\cdot,\cdot\}$ defines a
nondegenerate hermitian form on $\u^-$. By Theorem 2.3.13 of \cite{Kumar}, this form is positive
definite.
 Extend
$\{\cdot ,\cdot \}$ to
$\w\u^-$ in the usual way:  elements in  $\w^r\u^-$
  are  orthogonal to elements of $\w^s\u^-$ if $r\ne s$ whereas

$$
(X_1\wedge\dots\wedge X_r,Y_1\wedge\dots\wedge
Y_r)=\det\left(\,\{X_i,Y_j\}\,\right).
$$
Similarly, we can extend $(\cdot,\cdot)$ to define a symmetric bilinear form on $\w\widehat
L(\g,\sigma)$. If we extend $\s_o$ to $\w^k\widehat L(\g,\sigma)$ by setting $\s_o(x^1\w\cdots\w
x^k)=\s_o(x^1)\w\cdots\w \s_o(x^k)$, then obviously
\eqref{sigma} still holds  with $x,y\in\w\u^-$.

\vskip10pt
Set $\dd_p:\w^p\u^-\to\w^{p-1}\u^-$ to be the usual Chevalley-Eilenberg boundary operator
defined by 
$$
\dd_p(X_1\w\ldots\w X_p)=\sum_{i<j}(-1)^{i+j}[X_i,X_j]
\wedge X_1\dots\widehat{X}_i\dots\widehat{X}_j\dots\wedge
X_{p}
$$
if $p>1$ and $\dd_1=\dd_0=0$ and let $H_p(\u^-,\C)$ be its homology.
Let $L_p:\w^p\u^-\to\w^{p}\u^-$ be the corresponding Laplacian:
$$
L_p=\dd_{p+1}\dd^*_{p+1}+\dd^*_{p}\dd_p.
$$
 where $\dd_p^*$ denotes the adjoint of $\dd_p$ with respect to $\{\cdot ,\cdot \}$.
 \par
We shall use the following two basic properties of $L_p$ (see e.g. \cite[\S~2]{Kveryold})
\begin{equation}\label{iso} Ker\,L_p\cong H_p(\u^-).
\end{equation}
\begin{equation}\label{im} (Ker\,L_p)^\perp=Im\,\partial^*_p+Im\,\partial_{p+1}.
\end{equation}

Since $\u^-$ is stable under $ad(\m)$ we have an action of $\m$ on $\u^-$. Restricting this action to $\k$ we get an action of $\k$ on $\u^-$. Notice also
that, since $c$ is a central element, the action of $c$ on $\u^-$ is trivial. Recall that the Casimir 
operator $\Omega_{\k}$ of $\k$ is the element of the universal enveloping algebra of $\k$ defined by
setting 
$$
\Omega_{\k}=\sum_{i=1}^Nb_ib'_i,
$$
 where
$\{b_1,\ldots,b_N\},\,\{b'_1,\ldots,b'_N\}$ are dual bases of $\k$ with respect to
$(\cdot,\cdot)$. Set 
$\{u_1,\ldots,u_{n}\}$ and $\{u^1,\ldots,u^{n}\}$ to be bases of
$\h$ dual to each other with respect to $(\cdot,\cdot)$. It is well known that  $\Omega_{\k}$ can be rewritten as 
$$
\Omega_{\k}=\sum_{i=1}^{n}u_iu^{i}+2\nu^{-1}(\rho_0)+\sum_{\a\in\Dp_0}x_{-\a}x_\a.
$$
where $\rho_0=\frac{1}{2}\sum\limits_{\a\in\Dp_0}\a$ and $x_\a$ is a root vector in $\k$.\par
Define $\L_0\in\ha^*$ by setting $\L_0(\h_0)=\L_0(d)=0$ and $\L_0(c)=1$. Set  
\begin{equation}\label{rho}
\rho=\half\L_0+\rho_0
\end{equation}
Notice that, since $\s^2=Id$ and $(\cdot,\cdot)$ is the Killing form, then $\rho$ coincides with the element
$\rhat_\s$ defined in \cite[(4.27)]{kac-2007}. In particular, by \cite[Lemma~5.3]{kac-2007}), we have that  
$\rho(\a_i^\vee)=1$ for $i=0,\dots,n$ and $\rho(d)=0$.
\par
\section{The results}
First we make explicit the relationship between the Casimir element and the Laplacian.
\begin{prop}\label{laplacian} For $x\in\w\u^-$ we have 
\begin{equation}\label{lap}
L_p(x)=-\half(d+\Omega_{\k})(x).
\end{equation}
\end{prop}
\begin{proof}
Note that 
$\{u_1,\ldots,u_{n},c,d\}$ and $\{u^1,\ldots,u^{n},d,c\}$ are bases of
$\ha$ dual to each other with respect to $(\cdot,\cdot)$.  Then, following \cite{Kumar}, we set
$$\Omega=\sum_{i=1}^{n}u_iu^{i}+2cd+2\nu^{-1}(\rho)+\sum_{\a\in\Dp_0}x_{-\a}x_\a.$$
By \eqref{rho}, we have that 
$$
\Omega=\Omega_{\k}+d+2dc.
$$
   
The Laplacian calculation done at p.105 of \cite{Kumar} applied to $\w \u^-\simeq\C\otimes \w\u^-$
gives that, if
$
x\in\w\u^-$, then
$
L_p(x)=-\half \Omega(x).
$
Hence, by observing that $c$ acts trivially on $\w\u^-$, the result follows.
\end{proof}

We need to recall a key tool in what follows, namely Garland-Lepow\-sky generalization of 
Kostant's
theorem on the cohomology of the nilpotent radical.  We need some more notation.
 If $\lambda\in \ha^*$ is such that $\ov \l$ is dominant integral for $\Dp_0$, denote by 
$V(\lambda)$ be the irreducible $\m$-module of highest weight $\lambda$. Denote by $\Wa$ the Weyl group of $\widehat L(\g,\si)$. If $w\in\Wa$ set 
$$N(w)=\{\beta\in \Dap\mid w^{-1}(\beta)\in -\Dap\}.$$
Set 
$$\Wa'=\{w\in\Wa\mid w^{-1}(\Dp_0)\subset \Dap\}.$$ 
The following is a special case of Theorem 3.2.7 from \cite{Kumar}, which is an
extended version of  Garland-Lepowsky result \cite{Ga}.
\begin{theorem}\label{GL}
$$H_p\left(\u^-\right)=\bigoplus_{\substack{w\in \Wa'\\ \ell(w)=p}} V\left(w(\rho)-
\rho\right)
.$$
Moreover a representative of the highest weight vector of $V\left(w(\rho)-\rho\right)$
is given by 
\begin{equation}\label{hw}X_{-\beta_1}\wedge\dots\wedge X_{-\beta_p}\end{equation} 
where $N(w)=\{\beta_1,\ldots,\beta_p\}$  and
the
$X_{-\beta_i}$ are root vectors in $\wL(\g,\s)$. 
\end{theorem}
\vskip15pt
We now define
\begin{equation}\label{grado}
\w^{(r,s)} \u^{-}=Span\left\{ x^1_{i_1}\wedge
x^2_{i_2}\wedge\dots\wedge x^r_{i_r}\mid
 - \sum_{i=1}^r i_{j}=s \right\}.
\end{equation}
Note that the map $x^1_{-\frac{1}{2}}\w\ldots\w x^r_{-\frac{1}{2}}\mapsto
x^1\w\ldots\w  x^r$ affords   a canonical
identification
\begin{equation}\label{Ident}
\w^{(r,r/2)} \u^{-} \ {\buildrel \cong \over \longrightarrow}\
\w^{r}\pp
\end{equation}
that intertwines the adjoint action of $\k$.
\begin{rem}\label{spin} Recall that there is a standard linear isomorphism $\tau: so(\pp)\to \w^2\pp$ given by 
$\tau(\varphi)=-\frac{1}{4}\sum_i\varphi(p_i)\w p^i$, where $\{p_i\},\,\{p^i\}$ are dual basis of
$\pp$ with respect to 
$(\cdot,\cdot)_{|\pp}$. The adjoint action $ad_\pp$ of $\k$ on $\pp$ defines an embedding $\theta:\k\to so(\pp)$.
Observe that   $Im\,\tau\circ\theta$  corresponds, under the identification \eqref{Ident}, to $\partial^*_2(\w^{(1,1)}\u^-)$.
Infact, for $x\in\k$, a formal calculation affords

$$\partial^*_2(x_{-1})=-\frac{1}{2}\sum_{t=1}^{\dim\pp} [x,p_t]_{-\half}\w 
p^t_{-\half}.$$
\end{rem}
\medskip

\begin{lemma}\label{lemma 1}
Given linearly independent elements $x^1,\dots,x^p$ of $\pp$, set $v=x^1_{-\frac{1}{2}}\w\ldots\w x^p_{-\frac{1}{2}}$. Then $\dd_p(v)=0$ if and
only if $[x^i,x^j]=0$ for all $i,j$.
\end{lemma}
\begin{proof}
This follows readily from the definition of $\dd_p$:
$$
\dd_p(v)=\sum(-1)^{i+j}[x^i,x^j]_{-1}\w x^1_{-\half} \dots\widehat{x^i_{-\half}}\dots\w 
\widehat{x^j_{-\half}}\dots\w x^p_{-\half}.
$$
\end{proof}
For a $p$-dimensional subspace
$\aa=\bigoplus\limits_{i=1}^p\mathbb C v^i$
 of
$\pp$ define
\begin{align*}
&v_\aa=v^1\w\ldots\w v^p\in\w^p\pp,\\
&\widehat v_\aa=v^1_{-\half}\w\ldots\w v^p_{-\half}\in\w^{(p,p/2)}\u^-.\end{align*}

\begin{theorem}\label{eigen} The maximal eigenvalue for the action of $\Omega_\k$ on $\w^p\pp$ is at most  $p/2$.
Equality holds if and only if there exists a commutative subspace $\aa$ of $\pp$ of dimension $p$.
In such a case, $v_\aa$ is an eigevector for $\Omega_\k$ relative to the eigenvalue $p/2$. \end{theorem}
\begin{proof} To prove the first statement, remark that $L_p$ is self-adjoint and positive semidefinite on $\w \u^-$ with respect
to $\{\,,\,\}$. Since, by Proposition \ref{laplacian},  $\Omega_\k=-d-2L_p$, the claim follows.\par
Suppose that $\aa$ is an abelian subspace of $\pp$ of dimension $p$. Then, by Lemma \ref{lemma 1}, 
$\partial_p(\widehat v_\aa)=0$. Since $\widehat v_\aa\in\w^{(p,p/2)}\u^-$, we have that 
$\partial^*_{p+1}(\widehat v_\aa)=0$, hence $L_p(\widehat v_\aa)=0$. Therefore,  by \eqref{lap}, we have
$\Omega_\k(v_\aa)=p/2\,v_\aa$. Conversely, if $\Omega_\k$ has eigenvalue $p/2$ on $\w^p\pp$, then
$Ker\,L_p\cap\w^{(p,p/2)}\u^-\ne 0$. Using \eqref{iso} and  Theorem \ref{GL}, 
we know that $Ker\,L_p$ decomposes with multiplicity one. Since $\w^{(p,p/2)}\u^-$ is $\mathfrak m$-stable, 
we deduce that one the highest weight vectors \eqref{hw},
 say $x^1_{-\half}\w\cdots\w x^p_{-\half}$,  must belong to $
Ker\,L_p\cap\w^{(p,p/2)}\u^-$.  Since $\partial^*_p\partial_p=0$ implies that $\partial_p=0$, Lemma \ref{lemma 1} gives that
$Span(x^1,\ldots,x^p)$ is the required  abelian subspace.
\end{proof}
\vskip5pt

We now  relate the vectors $v_\aa$ to distinguished elements of $\Wa$. Set $\D_\pp$ to be the set of $\h_0$-weights
of
$\pp$ and suppose
 that $\i$ is a $\h_0$-stable subspace of $\pp$. Set
\begin{align*}
&\Phi_\i=\{\a\in\D_\pp\mid \pp_\a\subset\i\},\\  
&\widehat\Phi_\i=\{\tfrac{1}{2}\d-\a\mid\a\in\Phi_\i\}.
\end{align*}

\begin{theorem}\label{w}The following statements are equivalent
\item 1)
$\i$ is an abelian $\b_0$-stable subspace of $\pp$.

\item 2)
There is an element $w_\i\in \widehat W$ such that $N(w_\i)=\widehat\Phi_\i$.

\item 3)
$\i$ is a $\b_0$-stable subspace of $\pp$ and $\Omega_\k v_\i=\half(dim\,\i) v_\i$.
\vskip5pt\noindent
\end{theorem}
\begin{proof}
$1)\implies 2)$. Set $p=\dim\i$. Then, since $\i$ is abelian, $\dd_p(\widehat
v_\i)=0$.
Notice that $\widehat v_\i\in\w^{(p,p/2)}\u^-$, so $\dd_p^*(\widehat
v_\i)=0$. It follows that $L_p(\widehat v_\i)=0$. Since $\i$ is $\b_0$-stable,
$\widehat v_\i$ is a maximal vector for $\m$ in $\w\u^{-}$. By
Theorem \ref{GL}, there is an element  $w_{\i}\in\widehat W$ such that
$$
\w_{\a\in N(w_{\i})}X_{-\a}=\widehat v_{\i}
$$
and this implies that $N(w_{\i})=\widehat\Phi_{\i}$.
\par
$2)\implies 3)$. By Theorem \ref{GL} we have that  $\widehat v_{\i}$ is a maximal vector for the
action of
$\m$ on $\w\u^{-}$, hence $\i$ is a $\b_0$-stable subspace of $\pp$. Moreover $L_{p}(\widehat
v_{\i})=0$ therefore
$$
\Omega_{\k}(\widehat v_{\i})=-(2L_{p}+d)(\widehat v_{\i})=\half(dim\,\i)\widehat v_{\i},
$$
and this implies that $\Omega_{\k}v_{\i}=\half(dim\,\i) v_{\i}$.
\par
$3)\implies 1)$. This follows from Theorem \ref{eigen}. \end{proof}
\vskip15pt
 Let 
$\widehat A_p$ denote the linear span of the vectors $\widehat v_\aa$ when $\aa$ ranges over the
set of commutative subalgebras of $\pp$ of dimension $p$. Let also $\widehat M_p$ denote
the eigenspace corresponding to the eigenvalue $p/2$ for the action of $\Omega_\k$ on
$\w^{(p,p/2)}\u^-$.\par
 Denote by $\aa_1,\ldots,\aa_r$ the abelian
$\b_0$-stable subspaces of $\pp$ of dimension $p$ and set 
$\mu_i=\sum_{\a\in\widehat\Phi_{\aa_i}}\a=-\half\dim(\aa_i)\d+\sum_{\a\in\p_{\aa_i}}\a$. \par
Set $\widehat J$ to be the ideal (for exterior multiplication) in
$\w\u^-$ generated by $\dd_2^*(\u^-)$ and set $\widehat J_p=\widehat J\cap\w^{(p,p/2)}\u^-$.
\par 

\begin{prop}\label{aff}
\item 1)  $\widehat A_p=\widehat M_p=\bigoplus\limits_{i=1}^r V(\mu_i)=Ker(L_p)$.

\item 2) 
$$\w^{(p,p/2)}\u^-=\widehat A_p\oplus \widehat J_p$$
is the orthogonal decomposition with respect to  $\{\cdot ,\cdot \}$. In particular, letting 
$\mathcal A$ be the subalgebra of $\bigoplus\limits_{p\geq 0}\w^{(p,p/2)}\u^-$ generated by $1$ and $\dd_2^*(\u^-)$ then 
$$\bigoplus_{p\geq 0}\w^{(p,p/2)}\u^-=\mathcal A\w\sum_{p\geq 0} \widehat A_p$$
\end{prop}
\begin{proof}
1).  By Theorem \ref{eigen}, the linear generators of $\widehat A_p$ are eigenvectors for $\Omega_\k$
of eigenvalue $p/2$, hence $\widehat A_p,\subseteq \widehat M_p$. Clearly, by \eqref{lap}, $\widehat M_p\subseteq Ker\,L_p$.
For any element $w\in\Wa$, the following relation holds (see e.g. \cite[Corollary 1.3.22]{Kumar}):
$$w(\rho)-\rho=-\sum_{\a\in N(w)} \a.$$   
Combining this observation with  Theorem \ref{GL} and Theorem \ref{w}, we have that
$Ker\,L_p=\bigoplus\limits_{i=1}^r V(\mu_i)$. Finally, by Theorem \ref{w}, $V(\mu_i)$ is linearly
generated by elements in $\widehat A_p$, hence 
$\bigoplus\limits_{i=1}^r V(\mu_i)\subseteq \widehat A_p$.
\item 2). We have
$$\widehat A_p^\perp=(KerL_p)^\perp=\dd_p^*(\w^{(p-1,p/2)}\u^-).$$
The first equality is clear from part 1), whereas the second follows  combining \eqref{im} with the fact that
$\w^{(p+1,p/2)}\u^-=0$. It remains to prove that $\dd_p^*(\w^{(p-1,p/2)}\u^-)=\widehat J_p$.
Observe that, if $v\in\w^{(p-1,p/2)}\u^-$, then necessarily $v$ is a sum of decomposable elements of  type
$x^1_{-1}\w x^2_{-\half}\w\cdots\w x^{p-1}_{-\half}$.  Assume that $v=x^1_{-1}\w x^2_{-\half}\w\cdots\w x^{p-1}_{-\half}$. Since
$\partial^*$ is a skew-derivation and
$\partial^*_{p-1}(x^2_{-\half}\w\cdots\w x^{p-1}_{-\half})\in\w^{(p-1,\frac{p-2}{2})}\u^-=0$, we have
$$\dd_p^*(v)=\dd_2^*(x^1_{-1})\w x^2_{-\half}\w\cdots\w x^{p-1}_{-\half},$$
so that $\dd_p^*(v)\in \widehat J_p$. Conversely, if $w\in \widehat J_p$, then $w=\partial_2^*(x)\w y$ with
$x\in\w^{(1,s)}\u^-,\, y\in\w^{(p-2,r)}\u^-$. Since $s+r=p/2,\,r\geq \frac{p-2}{2},\,s\geq 1$, we have necessarily $s=1,r=
\frac{p-2}{2}$. Therefore $\partial^*_{p-1}(y)=0$, hence $w=\partial^*_p(x\w
y)\in\dd_p^*(\w^{(p-1,p/2)}\u^-)$.\par Finally, if $x\in\bigoplus_{p\geq 0}\w^{(p,p/2)}\u^-$, then
$x=a_1+\partial_2^*(j_1)\w b_1$  with $a_1\in \widehat A_p,\,j_1\in
\u^-,\, b_1\in\w^{(p-2,\frac{p-2}{2})}\u^-$. In turn, we can write $b_1=a_2+\partial_2^*(j_2)\w b_2$ as above, and so on. The last
claim now follows.
\end{proof}

Using the map \eqref{Ident}, the previous Proposition can be restated  as a result on the algebra
$\w\pp$. We set  
$A_p$ to be the linear span of the vectors $v_\aa$ when $\aa$ ranges over the
set of commutative subalgebras of $\pp$ of dimension $p$,  $M_p$ to denote
the eigenspace corresponding to the eigenvalue $p/2$ for the action of $\Omega_\k$ on
$\w^p\pp$.
 Denote by $L(\xi)$ the irreducible $\k$-module of highest
weight $\xi$.  \par
Set $J$ to be the ideal (for exterior multiplication) in
$\w\pp$ generated by $(\tau\circ\theta)(\k)$ and set $J_p=J\cap\w^{p}\pp$.
\par

\begin{theorem}\label{finito}
\item 1)  $A_p=M_p=\bigoplus\limits_{i=1}^r L(\sum\limits_{\a\in\p_{\aa_i}}\a)$.

\item 2) 
\begin{equation}\label{orth}\w^{p}\pp=A_p\oplus J_p\end{equation}
is the orthogonal decomposition with respect to the form on $\w\pp$ defined by extending, by determinants, the Killing form of
$\g$. Moreover, letting 
$A$ be the subalgebra of $\w\pp$ generated by $1$ and $(\tau\circ\theta)(\k)$ then 
$$\w\pp= A\w\sum_{p\geq 0} A_p.$$
\end{theorem}
\begin{proof} The only statement which does not follows directly from \eqref{Ident} and  Remark \ref{spin} is that 
the decomposition \eqref{orth} is orthogonal with respect to the form induced by the Killing form. Fix $x\in J_p$ and 
$v_\aa=x^1\w\cdots\w x^p$ with $[x^i,x^j]=0$. We want to show that $(x,v_\aa)=0$. Let $\hat x$ be the element
of $\widehat J_p$ corresponding to $x$ under \eqref{Ident}. Then
\begin{equation}\label{eq}(x,x^1\w\cdots\w x^p)=(\hat x,x^1_\half\w\cdots\w x^p_\half)=\{\hat x,\s_o(x^1_\half\w\cdots\w
x^p_\half)\}.\end{equation} Set $\tilde x^i_{-\half}=\s_o(x^i_\half)$. We now observe that  $[\tilde x^i,\tilde x^j]=0$. Indeed
\begin{align*}&[\tilde x^i,\tilde x^j]_{-1}=\\ &[\tilde x^i_{-\half},\tilde x^j_{-\half}]=[\s_o(x^i_\half),\s_o(x^j_\half)]=
-\s_o([x^i_{\half},x^j_{\half}])=-\s_o([x^i,x^j]_1),\end{align*}
and the last term is zero since $\aa$ is abelian. Since \eqref{eq} gets rewritten as
$$(x,v_\aa)=\{\hat x,\tilde x^1_{-\half}\w\cdots\w
\tilde x^p_{-\half}\},$$
and $\hat x\in \widehat J_p$, Proposition \ref{aff} implies that $(x,v_\aa)=0$.
\end{proof}
\section{Remarks}
{\bf 1.} If $\mathfrak s$ is a simple Lie algebra, consider the semisimple algebra $\g=\mathfrak s\oplus \mathfrak s$,
endowed with the switch automorphism $\s(x,y)=(y,x)$. Then we have 
$\k\cong\pp\cong\mathfrak s$ and we recover  Kostant's classical
results on abelian ideals of Borel subalgebras.  Theorems \ref{eigen} appears in \cite{K1} as
Theorem 5.  The statements in Theorem  \ref{finito} appear
 in \cite[Theorem 8]{K1} and \cite[Theorems A, B]{K2}. In all cases  proofs
are different from the ones we have presented.
\par
Subsequently Kostant realized  the connection of abelian ideals with Lie algebra homology (see \cite{Kinv}):
our tratment is inspired by this approach.
\vskip5pt
\noindent{\bf 2.} In the graded case Theorem  \ref{eigen} appears in \cite[Theorem 0.3]{P}, whereas Theorem \ref{finito}
appears in \cite[Theorems 1.1, 1.2]{Han}. Both authors do not exploit the connection with Lie algebra homology.
\vskip5pt
\noindent{\bf 3.} The so-called Peterson's  abelian ideals theorem states that the number of abelian ideals of a Borel
subalgebra of $\g$ in $2^{\text{rank}\,\g}$. This result shed a new light on the  results from \cite{K1}, as Kostant pointed out
in \cite{Kimrn}.   This paper contains an outline of proof of Peterson's result and a proof of 
equivalence $1)\Leftrightarrow 2)$
of Theorem \ref{w} for abelian ideals (see \cite[Section 2]{Kimrn}).\par
 A proof of Peterson's theorem using the geometry of alcoves is given in \cite{CP}. Combining this geometric approach with 
Garland-Lepowsky theorem, a uniform enumeration of abelian $\b_0$--stable subspaces in $\pp$ has
 been obtained in \cite{cmp}. The proof of Theorem \ref{w} is taken from \cite[Theorem 3.2]{cmp}.
\vskip5pt\noindent

\providecommand{\bysame}{\leavevmode\hbox to3em{\hrulefill}\thinspace}
\providecommand{\MR}{\relax\ifhmode\unskip\space\fi MR }
\providecommand{\MRhref}[2]{%
  \href{http://www.ams.org/mathscinet-getitem?mr=#1}{#2}
}
\providecommand{\href}[2]{#2}

\vskip10pt
\footnotesize{

\noindent{\bf Pierluigi M\"oseneder Frajria}: Politecnico di Milano, Polo regionale di Como, 
Via Valleggio 11, 22100 Como,
ITALY;\\ {\tt frajria@mate.polimi.it}
\vskip5pt
\noindent{\bf Paolo Papi}: Dipartimento di Matematica, Universit\`a di Roma 
``La Sapienza", P.le A. Moro 2,
00185, Roma , ITALY;\\ {\tt papi@mat.uniroma1.it} }

\end{document}